\documentclass[english]{article}
\usepackage[T1]{fontenc}
\usepackage[latin9]{inputenc}
\usepackage{amsmath}
\usepackage{amssymb}
\usepackage{babel}
\usepackage{amsthm}
\usepackage{graphicx}
\usepackage{tikz}
\usepackage[all]{xy}
\usepackage{bm}
\usepackage{comment}
\usetikzlibrary{calc,patterns,angles,quotes}
\usepackage[toc,page]{appendix}
\usepackage{hyperref}

\newtheorem{theorem}{Theorem}
\newtheorem{lemma}{Lemma}

\newtheorem{defi}{Definition}
\newtheorem{criteria}{Criterion}

\newcommand{\R}{\mathbb{R}}
\newcommand{\T}{\mathbb{T}}

\newcommand{\C}{\mathbb{C}}

\newcommand{\N}{\mathbb{N}}

\title{{Concave Toric Domains in Stark-type Mechanical Systems}}
\author{Airi Takeuchi, Lei Zhao}

\begin{document}

\maketitle

\begin{abstract}
{It has been shown {in \cite{Frauenfelder}} that the bounded component of the energy surface of the planer Stark problem after the Levi-Civita transformation are concave toric domains. In this paper, we present a different approach on the problem of determining concave toric domains in a family of integrable natural mechanical problems in the plane based on the computation of action-variables. We give criteria on the potentials for the bounded components of the energy hypersurfaces to be concave toric domains and apply these criteria to a class of problems.}
\end{abstract}

\section{Introduction} 
In symplectic geometry, interactions between the geometry and the dynamics of Hamiltonian systems are investigated. Many surprising results are unveiled. For a two degree of freedom integrable Hamiltonian system with a Hamiltonian $\T^{2}$-symmetry, its energy hypersurfaces are characterized equivalently by curves in the non-negative quadrant $(I_{1} \ge 0, I_{2} \ge0)$ of $\R^{2}$ connecting the axes, as the image of the corresponding moment map. The geometric properties of these curves are of deep relevance to their symplectic geometrical properties. The corresponding energy hypersurface is called a \emph{concave toric domain}, if this curve is the graph of a smooth non-increasing convex function. It is called a \emph{convex toric domain}, if this curve is the graph of a smooth non-increasing concave function. Interests on these special types of energy hypersurfaces comes from recent developments on symplectic embedding problems with the embedded contact homology (ECH) \cite{Hutchings}: In \cite{DCG}, it is shown that the ECH capacities are sharp when the symplectic embedding problem from the interior of a concave toric domain to a convex toric domain is concerned. The examples of concave toric domain has been studied in \cite{Ramos} for Lagrangian bidisks, in \cite{Frauenfelder} for the Stark problem, and in \cite{Pinzari} for the two-center problem. Elliptic functions and elliptic integrals are extensively used in the computations of the latter two references.

This note continues and complements the work \cite{Frauenfelder}. The Stark problem is among a class of natural mechanical systems that are simultanously regularizeable and separable with the Levi-Civita regularization. For these systems we shall establish criteria for the problem of determining whether the bounded components of the energy hypersurfaces are convex/concave domains, based on the computation of action-angle variables. These criteria do not apply in the case of Stark problem. Nevertheless, they can be applied in a simple fasion to the system called \emph{frozen-Hill's problem with centrifugal correction}, whose potential is given as 
\begin{equation}\label{eq: potential frozen hill c c}
V(q)=-\dfrac{m}{|q|}- gq_1^2 - g\frac{q_2^2}{4}, \qquad m >0,\, g \neq 0.
\end{equation}
{For this potential, there exist two different critical values when $g>0$ and there is no critical values when $g<0$. We call the lower one the first critical value. }
For this system we have

\begin{theorem}\label{thm: 1}
	In the frozen-Hill's problem with centrifugal correction with potential given by \eqref{eq: potential frozen hill c c} with $g>0$, for energy below the first critical value, after Levi-Civita regularization, the corresponding bounded component of the energy hypersurface bounds a concave toric domain. 
\end{theorem}

We also have
\begin{theorem}\label{thm: 2}
	In the frozen-Hill's problem with centrifugal correction with potential given by \eqref{eq: potential frozen hill c c} with $g<0$, any negative energy-hypersurface after Levi-Civita regularization bounds a convex toric domain. 
\end{theorem}

We consider also systems with potentials
\begin{equation}\label{eq: more gen pot}
V(q)=-\frac{m}{|q|} - g\sum_{k= 0}^{n} \frac{C_k}{2^{2 k }} q_1^{2 n-  2 k } q_2^{2k},\quad  m >0,  g \neq 0, n \in \N
\end{equation}
where the coefficients $\{C_k\}_{k = 0}^n$ are defined inductively as 
\begin{equation}
\label{eq: Cn_definition}
C_0 = 1, \quad
C_k = (-1)^{k}- \left(\sum_{\ell=0}^{k-1} (-1)^{k-\ell}  
\begin{pmatrix}
2n -2 \ell\\
k- \ell 
\end{pmatrix}
C_{\ell}
  \right).
\end{equation}

We have
\begin{theorem}\label{thm: 3}
	In the system with potential given by \eqref{eq: more gen pot} with $g>0$, for energy below the first critical value,
	 after Levi-Civita regularization, the corresponding bounded component of the energy hypersurface bounds a concave toric domain, as long as $G_{k}=4 C_{k} (n-k) - C_{k+1} (k+1)$ are all positive for $k=1,\cdots, n$.
	 
\end{theorem}

We verify that the last condition is satisfied for $n=1,2,3,4,5$. We do not yet know whether this holds for all $n$.

\begin{theorem}\label{thm: 4}
	In the system with potential given by \eqref{eq: more gen pot} with $g<0$, any negative energy-hypersurface after Levi-Civita regularization bounds a convex toric domain. 
\end{theorem}

\section{Separability of Stark-type problems}
We consider a class of mechanical systems in $\R^2$ defined with potential given in the form
\begin{equation}
	\label{eq: Stark_type_force_function}
{V(q):}=-\frac{m}{|q|} - G(q_1,q_2), \quad m >0,  \quad G \in C^{\infty}(\R^2 \setminus O, \R),
\end{equation}
which is a modification to the Kepler potential.

{It is known that the singularity of the Kepler problem can be regularized via the Levi-Civita regularization \cite{Levi-Civita}, which extends naturally to the type of problem under consideration. We recall this procedure.}

The Hamiltonian function of the system is 
\[
{H(p, q)}= \frac{|p|^2}{2}- \frac{m}{|q|} - G(q_1,q_2).
\]

{We consider its  $(-f)$-energy level $\{H=-f\}$, which is the $0$-energy level of the shifted Hamiltonian $H+f$:
\[
\Bigl\{H +f = \frac{|p|^2}{2}- \frac{m}{|q|} - G(q_1,q_2) + f = 0\Bigr\}.
\]}
The contangent lift of the complex square mapping
\[
{\C \setminus O \times \C \mapsto \C \setminus O \times \C}, \quad (z,w)\mapsto (q = z^2, p=\frac{w}{2 \bar{z}})
\]
is canonical, and pulls the shifted Hamiltonian $H+f$, together with its $0$-level, back to
\[
{\Bigl\{\frac{|w|^2}{8 |z|^2} - \frac{m}{|z|^2} - G(z_1^2 - z_2^2, 2z_1 z_2)+f=0.\Bigr\} \subset \C \setminus O \times \C}
\]
{Multiplication of the Hamiltonian function appearing on the left-hand side by a positive factor only reparametrizes the flow on this energy-level, since this modification does not change the hypersurface and in particular the kernel-distribution of the restricted symplectic form on it.}

{Now with the time reparametrization by multiplying the above Hamiltonian function by the factor $|z|^{2}$ we obtain
\[
K = \frac{|w|^2}{8} -m - (z_1^2+z_2^2)G(z_1^2 - z_2^2, 2z_1 z_2)+f(z_1^2+z_2^2) = 0.
\]
The transformed Hamiltonian system $K$ is regular if the function 
$$(z_1^2+z_2^2)G(z_1^2 - z_2^2, 2z_1 z_2)$$
can be extended to a smooth function in $(z_{1}, z_{2})$. Moreover, the system is separable if the function $(z_1^2+z_2^2)G(z_1^2 - z_2^2, 2z_1 z_2)$ is separable in $(z_1, z_2)$ coordinates. This motivates the following definition:}
\begin{defi}\label{def: 1}
	A {natural mechanical system defined in the plane} with a potential in the form \eqref{eq: Stark_type_force_function} is called {of} \emph{Stark-type} {if there exists functions $G_1, G_2 \in C^{\infty}(\R, \R)$ such that
	  \[(z_1^2+z_2^2)G(z_1^2 - z_2^2, 2z_1 z_2) = G_1(z_1) + G_2(z_2)\]
        	holds.}
\end{defi}


In the following we {briefly discuss some} important examples of Stark-type problems. 

\paragraph{Kepler Problem} The case $G=0$ is the Kepler problem. The Levi-Civita regularization procedure relates it to the Hooke problem in the plane.

\paragraph{Stark Problem} The Stark problem corresponds to the case  $G = g \, q_1$. where $0 \neq g \in\R$. This can be considered as a planar system with a gravitational field and an additional external constant force field. The shifted Hamiltonian of this system is expressed as 
\[
H + f = \frac{|p|^2}{2}- \frac{m}{|q|} - g q_1+f=0
\]
on its $-f$-energy hypersurface. The transformed Hamiltonian is separable: 
\[
K= \frac{|w|^2}{8} -m - g(z_1^4-z_2^4)+f(z_1^2+z_2^2) = 0,
\]
\paragraph{Frozen Hill's problem with centrifugal correction}

{The case with 
$$G=gq_1^2 + gq_2^2/4, \qquad g>0$$
describes the so-called Frozen Hill's problem with a repulsive centrifugal correction. Similarly, when $g <0$, it gives rise to the Frozen Hill's problem with attracting centrifugal correction.  The Hamiltonian of such a system is given by 
\[
H = \frac{|p|^2}{2}- \frac{m}{|q|} - gq_1^2 - g\frac{q_2^2}{4}, 
\]
The above-mentioned transformation on its $(-f)$-energy hypersurface transforms the system into the $m$-energy level of the Hamiltonian $K$:
\begin{equation}\label{eqn: K=m}
K= \frac{|w|^2}{8}  - g(z_1^6+z_2^6)+f(z_1^2+z_2^2) = m,
\end{equation}
after a time parameter change. This system is clearly separable in $(z_{1}, z_{2})$-coordinates. Moreover, the two separated Hamiltonians are of the same form. } 

\paragraph{More examples}
Further examples can be constructed by properly choosing the smooth functions $G_{1}, G_{2}$ in Definition \ref{def: 1}. This gives infinitely many Stark-type problems in the plane.



\section{Conditions on Concavity/Convexity}
We take a Hamiltonian of the form 
\[
K = \frac{|w|^2}{2}+V_1(z_1^2) + V_2(z_2^2) ,
\]
and we consider its $m$-energy hypersurface $\{K=m\}$, where $V_1, V_2, \in C^{\infty}(\R,\R) $. We write 
\[
K = K_1(z_1, w_1) + K_2(z_2,w_2),
\]
where 
$$K_1 = \frac{w_1^2}{2} + V_1(z_1^2), \qquad K_2 = \frac{w_2^2}{2} + V_2(z_2^2)$$ 
are {the decoupled Hamiltonians in the separation. }

We take {$a,b \in \R$} so that $a + b = m$. We assume that the level sets {$\{K_1=a\}$ and $\{K_2= b\}$} are all closed simple curves in $T^{*} \R \cong \R^{2}$. The {action variables are} computed as 
\[
I_1(a) := \frac{1}{2 \pi }\int_{K_1 =a} w_1 dz_1
\]
and 
\[
I_2(b) := \frac{1}{2 \pi }\int_{K_2 =b} w_2 dz_2.
\]
{From Stoke's theorem, $I_1$ and $I_2$ are the areas enclosed by the curves $\{K_1 =a\}$ and $\{K_2 =b\}$ respectively, which are strictly increasing with respect to $a$ and $b$ respectively.

{This monotonicity enables us to take the inverse functions, which are also strictly increasing:
\[
a = h_1(I_1), \quad b = h_2 (I_2).
\]
The condition $a + b = m$ translates into
\[
h_2 (I_2) = m - h_1(I_1).
\]
This determines $I_2$ as a smooth function of $I_1$. Consequently, by taking derivative we obtain 
\[
h_2'(I_2) \frac{d I_2}{d I_1} = - h_1'(I_1),
\]
which leads to 
\[
\frac{d I_2}{d I_1} = - \frac{h_1'(I_1)}{h_2'(I_2)}.
\]
which is always negative by the motononicity of the functions $h_{1}(I_{1}), \, h_{2}(I_{2})$.}

From this, we obtain 
\begin{align*}
\frac{d^2 I_2}{d I_1^2} &= -\frac{h_1''(I_1)h_2'(I_2) - h_1'(I_1) h_2''(I_2) \cdot  dI_2/dI_1 }{h_2'(I_2)^2}\\
&=-\frac{h_1''(I_1) h_2'^2(I_2)  + h_1'^2(I_1) h_2''(I_2) }{h_2'(I_2)^3}.
\end{align*}

Since $h_2'(I_2) = \frac{1}{I_2'(b)}>0$, we obtain 
\begin{equation}\label{eq: condition 111}
\frac{d^2 I_2}{d I_1^2} \gtrless 0 \iff h_1''(I_1)  h_2'^2(I_2) + h_1'^2(I_1) h_2''(I_2) \lessgtr 0.
\end{equation}

Thus, we have in particular

\begin{criteria} $I_{2}$ is a convex function of $I_{1}$ if 
$$h_1''(I_1)<0, \quad h_2''(I_2)<0.$$
Similarly,   $I_{2}$ is a concave function of $I_{1}$ if 
$$h_1''(I_1)>0, \qquad h_2''(I_2)>0.$$
\end{criteria}

From $$h_1''(I_1) = -\frac{I_1''(a)}{(I_1'(a))^3}, \qquad h_2''(I_2) = -\frac{I_2''(b)}{(I_2'(b))^3},$$
 we have equivalently

\begin{criteria} {$I_{2}$ is a convex function of $I_{1}$} if 
$${I''_{1}}(a)>0, \quad I_2''(b)>0.$$
 {$I_{2}$ is a concave function of ${I_{1}}$} if 
 $$I''(a)<0, \quad I_2''(b)<0.$$
\end{criteria}

{It is desirable to obtain criteria with conditions imposed only on the potentials $V_{1}$ and $V_{2}$. For this purpose we assume that the domain confined by the equation $K_{1}=a$ is star-shaped with respect to the origin. }
We introduce the following symplectic polar coordinates 
\[
w_1 = \sqrt{A_1} \cos \theta,\quad z_1 = \sqrt{A_1} \sin \theta. 
\]
{With the star-shape condition, plugging these coordinates into condition $K_{1}=a$ uniquely determines a function $A_{1}=A_{1} (a, \theta)$.}
In this coordinates, the action {variable} $I_1$ {is computed as}
\[
I_1 = \int_{0}^{2 \pi} \int_0^{A_1(a, \theta)} dA_1 d \theta{ =\int_{0}^{2 \pi} A_1(a, \theta) d \theta.}
\]
Thus we have $I_1''(a) \gtrless 0$ when {$\frac{d^2 A_1 (a, \theta)}{da^2} \gtrless 0$ holds} for almost all $\theta$. In the following, we denote $\frac{dA_1}{da}$ by $A'$ and $\frac{d^2A_1}{da^2}$ by $A''$.

We have
\[
\frac{A_1 \cos^2 \theta }{2} + V_1 (A_1 \sin^2\theta)= a.
\]
By differentiating with respect to $a$, we obtain 
\[
\frac{A_1' \cos^2 \theta }{2}+V_1'(A_1 \sin^2\theta) A_1'\sin^2 \theta= 1,
\] 
thus 
\[
A_1' = \frac{1}{{(\cos^2 \theta)}/2 + V_1'(A_1 \sin^2 \theta)\sin^2 \theta}.
\]
{Further,} we have
\[
\frac{A_1'' \cos^2 \theta }{2}+ V_1''(A_1 \sin^2\theta) A_1'^2 \sin^4 \theta+ V_1'(A_1\sin^2 \theta) A_1''\sin^2 \theta = 0{,}
\]

thus
\[
A_1''= - \frac{V_1''(A_1 \sin^2 \theta )\sin^4 \theta}{(\cos^2 \theta/2 + V_1'(A_1 \sin^2 \theta)\sin^2 \theta)^3}.
\]

From this, we have  $I_1''>0$ when 
$$V_1'(z_1^2)>0,\quad V_1''(z_1^2)<0.$$
Similarly, a sufficient condition for $I_2''>0$ is 
$$V_2'(z_2^2)>0, \quad V_1''(z_2^2)<0.$$

 {We summarize this discussion with the following Criterion:
\begin{criteria}\label{Cri: 3} $I_{2}$ is a convex function of $I_{1}$ if 
$$V_1'(z_1^2)>0, \,\, V_1''(z_1^2)<0 \quad V_2'(z_2^2)>0, \,\, V_1''(z_2^2)<0.$$

Similarly, $I_{2}$ is a concave function of $I_{1}$ if 
$$V_1'(z_1^2)>0, \,\, V_1''(z_1^2)>0 \quad V_2'(z_2^2)>0,\,\, V_1''(z_2^2)>0.$$
\end{criteria}
}

Note that all these criteria are far from being optimal. On the other hand, It is easy to check whether Criterion \ref{Cri: 3} is satisfied in concrete examples. When the concavity/convexity of the system is not determined by Criterion \ref{Cri: 3}, then more involved computation is inevitable. On the other hand, Criterion \ref{Cri: 3} gives open conditions. Together with Theorems \ref{thm: 1}, \ref{thm: 2} this shows the prolificacy of the existence of concave/convex toric domains in Stark-type systems.

{We now study concrete examples of Stark-type systems.}
\paragraph{Discussion on the Stark problem}
In the Stark problem we have 
$$V_1(z_1^2) = -gz_1^4 + fz_1^2, \quad V_2(z_2^2) = gz_2^4 + fz_2^2.$$
We compute
\[
V_1'(z_1^2) = -2g z_1^2 +f, \quad V_1''(z_1^2)= -4g 
\]
and
\[
V_2'(z_2^2) = 2g z_2^2 +f, \quad V_2''(z_2^2) = 4g.
\]
Thus $V_1''(z_1^2)$ and $ V_2''(z_2^2)$ are not of the same sign and therefore Criterion \ref{Cri: 3} does not determine the convexity/concavity of the energy hypersurfaces.

\paragraph{Frozen Hill's problem with centrifugal correction}
We consider now the frozen Hill's problem with centrifugal correction, in which we have
\[
V_1(z_1^2) = -gz_1^6 +f z_1^2, \quad V_2(z_2^2)= -g z_2^6 +f z_2^2.
\]
The regularized energy level is given in Eq. \ref{eqn: K=m}.

In this case, we obtain
\[
V_1'(z_1^2) = -3 g z_1^4 + f, \quad V_1''(z_1^2) = -6 g z_1^2
\]
and 
\[
V_2'(z_2^2) = -3 g z_2^4 + f, \quad V_2''(z_2^2) = -6 g z_2^2
\]

\textbf{Proof of Theorem \ref{thm: 1}}

We have $g>0$ in the hypothesis of Theorem \ref{thm: 1}, which implies $V_1''(z_1^2)<0$ and $V_2''(z_2^2)<0$.


The potential of the {unregularized} Frozen Hill's problem with centrifugal correction
\[
V := -\frac{m}{|q|} - g q_1^2 -g \frac{q_2^2}{4}
\]
has four critical points
$(0, \pm (\frac{g}{2m})^{-1/3})$ and $(\pm(\frac{2g}{m})^{-1/3}, 0 ) $, with critical values
$
E_2= -\frac{3}{2} \left(\frac{m^2 g}{2}\right)^{\frac{1}{3}}
$
and 
$
E_1 =-3 \left(\frac{m^2 g}{4}\right)^{\frac{1}{3}}
$
respectively.
We have
\[
E_1 < E_2 < 0.
\]
In the case $-f< E_{1}$, the corresponding energy hypersurface projects to a bounded component and an unbounded component in the configuration space. After regularization, the energy hypersurface $\{K=m\}$ consists of a bounded, closed hypersurface and an unbounded hypersurface.

We consider the function $V_{1} (z_{1}^{2})$, which has a unique maximum at $z_1^2= \sqrt{\frac{f}{3g}}$ with value $\frac{2f}{9}\sqrt{\frac{3f}{g}}. $ It follows from  $-f< E_{1}$ that $\frac{2f}{9}\sqrt{\frac{3f}{g}}>m$. Consequently in the bounded component of $\{K=m\}$ we have $z_{1}^{2}<\sqrt{\frac{f}{3g}}$, which implies $V_1'(z_1^2)>0$. Analogously we have $V_2'(z_1^2)>0$ on this bounded component. Theorem \ref{thm: 1}} now follows from Criterion \ref{Cri: 3}. \qed

\textbf{Proof of Theorem \ref{thm: 2}} 

We have $g<0, f >0$ in this case, the potential $V$ does not have any critical points. Then
$$V_{1}'(z_{1}^{2})>0,\, V_1''(z_1^2)>0,\,V_{1}'(z_{1}^{2})>0\, V_2''(z_2^2)>0.$$
 And the conclusion follows from Criterion \ref{Cri: 3}. \qed



\paragraph{More general Stark-type Systems}

We consider Stark-type Systems with potentials given by \eqref{eq: more gen pot}. 

Correspondingly we have 
$$G(z_1^2-z_2^2, 2z_1z_2)= -g\sum_{k = 0}^{2n} (-1)^{k} z_1^{2 k} z_2^{4n-2 k}.$$
And thus
\[
(z_1^2+ z_2^2)G(z_1^2-z_2^2, 2z_1z_2)= -g(z_1^2+z_2^2)\left(\sum_{k = 0}^{2n} (-1)^{k} z_1^{2 k} z_2^{4n-2 k}\right)  = -gz_1^{4n+2} - gz_2^{4n +2}.
\]
The regularized Hamiltonian is
\[
K = \frac{|w|^2}{8} + f (z_1^2+ z_2^2) - g (z_1^{4n+2} + z_2^{4n+2}).
\]
We consider its m-energy hypersurface$ \{K=m\}$.

We have $V_1(z_1^2)=f z_1^2- gz_1^{4n+2}$ and $V_2(z_2^2)=f z_2^2 - gz_2^{4n +2}$. Thus
\[
V'_1(z_1^2) = f - g(2n+1)z_1^{4n}, \quad V''_1(z_1^2)= -2gn(2n+1)z_1^{4n-2}
\]
and 
\[
V'_2(z_2^2) = f - g(2n+1)z_2^{4n}, \quad V''_2(z_2^2)= -2gn(2n+1)z_2^{4n-2}.
\]

\begin{lemma} {The generalized potential $V$ given by \eqref{eq: more gen pot} has four critical points 
	\[
	\left( \pm\left(m/(2ng)\right)^{\frac{1}{2n+1}},0\right),  \quad \left(0, \pm\left(2^{2n-1}m/ (ng)\right)^{\frac{1}{2n+1}}\right),
	\]
	with critical values
	\[
	E_1=-(2n+1)\left(\frac{m}{2n}\right)^{\frac{2n}{2n+1}} g ^{\frac{1}{2n+1}}
	\]
	and
	\[
	E_2 =-(2n+1)\left(\frac{m}{4n}\right)^{\frac{2n}{2n+1}} g ^{\frac{1}{2n+1}},
	\]
}
as long as $D_{k}:=4 C_{k} (n-k) - C_{k+1} (k+1)$ are all positive for $k=1,\cdots, n$.
\end{lemma}

\begin{proof} 
	{Let $(\tilde{q}_1, \tilde{q}_2)$ be a critical point of $V$. We first show that if $D_k$ are all positive for $k=1,\cdots, n$ then either $\tilde{q}_1=0$ or $\tilde{q}_2=0$.
	Set 
	\[F(q):= \sum_{k=0}^{n} \frac{C_k}{2^{2k}} q_1^{2n-2k} q_2^{2k}\]
	Then the potential $V$ is rewritten as
	\[
	V(q) = -\frac{m}{|q|} - g \cdot F(q)
	\]
	and its partial derivatives are computed as
	\begin{align*}
	&\frac{\partial V}{\partial q_1} = \frac{m \cdot q_1}{(q_1^2+q_2)^{3/2}} -g \cdot \frac{\partial F}{\partial q_1} \\
	&\frac{\partial V}{\partial q_2} = \frac{m \cdot q_2}{(q_1^2+q_2)^{3/2}} -g \cdot \frac{\partial F}{\partial q_2} 
	\end{align*}
Assume now that $\tilde{q}_1 \neq 0 $ and $\tilde{q}_2\neq 0$, then from $$\frac{\partial V}{\partial q_1} (\tilde{q}_1, \tilde{q}_2) = \frac{\partial V}{\partial q_2} (\tilde{q}_1, \tilde{q}_2) = 0$$ we obtain
\begin{align*}
&\frac{m }{(\tilde{q}_1^2+\tilde{q}_2)^{3/2}} -\frac{g}{\tilde{q}_1} \cdot \frac{\partial F}{\partial q_1} (\tilde{q}_1, \tilde{q}_2) = 0\\
&\frac{m }{(\tilde{q}_1^2+\tilde{q}_2)^{3/2}} -\frac{g}{\tilde{q}_2} \cdot \frac{\partial F}{\partial q_2} (\tilde{q}_1, \tilde{q}_2) = 0.
\end{align*}
From this, we deduce 
\begin{equation*}
\begin{split}
&\frac{1}{q_1} \cdot \frac{\partial F}{\partial q_1} (\tilde{q}_1, \tilde{q}_2)  - \frac{1}{q_2} \cdot \frac{\partial F}{\partial q_2} (\tilde{q}_1, \tilde{q}_2) \\
&=\sum_{k=0}^n \left(  \frac{C_k}{2^{2k}}(2n-2k) - \frac{C_{k+1}}{2^{2(k+1)}} 2(k+1) \right) \tilde{q}_1^{2n-2(k+1)} \tilde{q}_2^{2k}\\
&=0.
\end{split}
\end{equation*}

{This is however not possible if $4C_k(n-k) -C_{k+1}(k+1)=D_{k} >0$ for all $k = 1, \cdots, n$.}

{Thus at a critical point, either $\tilde{q}_{1}=0$ or $\tilde{q}_{2}=0$.}

We suppose $\tilde{q}_1 = 0${. T}hen the equation $$\partial V/\partial q_1 (\tilde{q}_1, \tilde{q}_2) =0$$ is satisfied for any $\tilde{q}_2 \neq 0$. The other equation $$\partial V/\partial q_2 (\tilde{q}_1, \tilde{q}_2) =0$$ is equivalent to 
\[
 \frac{m \cdot \tilde{q}_2}{(\tilde{q}_2)^{3/2}} - g \frac{C_n \cdot n \cdot \tilde{q}_2^{2n-1} }{2^{2n-1}} =0.
\]

{We claim $C_n=1$. Assuming this, the equation is easily solved} 

\[
\tilde{q}_2=\pm\left(2^{2n-1}m/ (ng)\right)^{\frac{1}{2n+1}}.
\]
The case $\tilde{q}_2 = 0$ are be treated similarly.
}

{{The fact {$C_{n}=1$} is established by the following computation:} From the inductive definition of $C_n$ given by \eqref{eq: Cn_definition}, we have
\begin{equation}
\label{eq: Cn_expression}
C_n= (-1)^n + (-1)^{n+1}(d_n + d_{n-1} + \cdots +d_1),
\end{equation}
where 
\[
d_1 =2, \quad
d_n:=
\begin{pmatrix}
2n\\
n
\end{pmatrix}
-
\begin{pmatrix}
2n\\
n-1
\end{pmatrix}
d_1
-
\begin{pmatrix}
2n\\
n-2
\end{pmatrix}
d_2
-\cdots
-
\begin{pmatrix}
2n\\
1
\end{pmatrix}
d_{n-1}.
\]
We show $d_n=(-1)^{n+1}2$ by induction. Suppose that $d_k = (-1)^{k+1}2$ for $k = 1, \cdots n-1$, then we have 
\begin{equation}
\begin{split}
d_n &= \begin{pmatrix}
2n\\
n
\end{pmatrix}
-
\begin{pmatrix}
2n\\
n-1
\end{pmatrix}
2
+
\begin{pmatrix}
2n\\
n-2
\end{pmatrix}
2
-\cdots
-
\begin{pmatrix}
2n\\
1
\end{pmatrix}
(-1)^{n}2\\
&=\left( \frac{2n(2n-1) \cdots (n+2)}{(n-1)(n-2) \cdots 1}  \right) \left( \frac{n+1}{n} -2 \right) +\begin{pmatrix}
2n\\
n-2
\end{pmatrix}
2
-\cdots
-
\begin{pmatrix}
2n\\
1
\end{pmatrix}
(-1)^{n}2
\\
&= - \frac{2n(2n-1) \cdots (n+2)}{(n-2)(n-3)\cdots 1}  +\begin{pmatrix}
2n\\
n-2
\end{pmatrix}
2
-\cdots
-
\begin{pmatrix}
2n\\
1
\end{pmatrix}
(-1)^{n}2 \\
&= \left(  \frac{2(2n-1) \cdots (n+3)}{(n-2)(n-3) \cdots 1} \right) (2n - (n+2))  - \begin{pmatrix}
2n\\
n-3
\end{pmatrix}
2
-\cdots
-
\begin{pmatrix}
2n\\
1
\end{pmatrix}
(-1)^{n}2 \\
&=- \frac{2n(2n-1) \cdots (n+3)}{(n-3)(n-4)\cdots 1}  -\begin{pmatrix}
2n\\
n-3
\end{pmatrix}
2
+\cdots
-
\begin{pmatrix}
2n\\
1
\end{pmatrix}
(-1)^{n+1}2 \\
& \ \vdots \\
&= (-1)^{n+1}  2.
\end{split}
\end{equation}
By plugging $d_n = (-1)^{n+1}2$ in \eqref{eq: Cn_expression}, we get $C_n=1$.
}
\end{proof}

\textbf{Proof of Theorem \ref{thm: 3}}

When $g>0$, we have 
$$V''_1(z_1^2)<0, \qquad V''_2(z_2^2)<0.$$
The maximum of the function is obtained at $z_1^{2}=(\frac{f}{g(2n+1)})^{1/2n}$ with value 
$$ f\left(\frac{f}{-g(2n+1)}\right)^{\frac{1}{2n}} + g \left(\frac{f}{g(2n + 1)}  \right)^{\frac{2n+1}{2n}}.$$
The condition that this value is larger than $m$ leads to 
$$ f > \left(\frac{m(2n+1)}{2n}\right)^{\frac{2n}{2n+1}} (g(2n+1) )^{\frac{1}{2n+1}} = -E_1.$$

With this condition we have $V''_1(z_1^2)>0$ and $V''_2(z_2^2)>0$ for points from the bounded component of the corresponding regularized energy level $\{K=m\}$. This shows that the bounded component of $\{K=m\}$ bounds a concave toric domain. \qed


\textbf{Proof of Theorem \ref{thm: 4}}
When $g<0$, $f>0$, we have
$$V_{1}'(z_{1}^{2})>0,\, V_1''(z_1^2)>0,\,V_{1}'(z_{1}^{2})>0\, V_2''(z_2^2)>0.$$
We conclude with Criterion \ref{Cri: 3}. \qed

\vspace{0.5cm}
\textbf{Acknowledgement} A.T. and L.Z. are supported by DFG ZH 605/1-1, ZH
605/1-2.

\hspace{-1cm}
\begin{tabular}{@{}l@{}}%
	Airi Takeuchi\\
	\textsc{University of Augsburg, Augsburg, Germany.}\\
	\textit{E-mail address}: \texttt{airi1.takeuchi@uni-a.de}
\end{tabular}
\vspace{10pt}

\hspace{-1cm}
\begin{tabular}{@{}l@{}}%
	Lei Zhao\\
	\textsc{University of Augsburg, Augsburg, Germany.}\\
	\textit{E-mail address}: \texttt{lei.zhao@math.uni-augsburg.de}
\end{tabular}

\end{document}